\documentclass[10pt, reqno]{amsart}
\usepackage{amsmath}
\usepackage{amssymb}
\usepackage{amsthm}
\usepackage[inline]{enumitem}
\usepackage{hyperref}
\usepackage{xcolor}
\usepackage{cleveref}
\usepackage[normalem]{ulem}
\usepackage{color}
\setlist[enumerate]{label={\upshape(\roman*)}}

\usepackage{amsrefs}
\setlist[enumerate]{label={\upshape(\roman*)}}
\newcommand{\bD}{\begin{definition}\label}
\newcommand{\eD}{\end{definition}}

\newcommand{\C}{\mathbb C}

\newcommand{\tr}{\operatorname{tr}}
\newcommand{\h}{\mathcal H}

\newcommand{\B}{\mathcal B}
\newcommand{\N}{\mathbb N}

\newcommand{\K}{\mathcal{K}}

\newtheorem*{theorem*}{Theorem}
\newtheorem{theorem}{Theorem}[section]

\newtheorem*{question*}{Question}
\newtheorem{definition}[theorem]{Definition}
\newtheorem{example}[theorem]{Example}
\newtheorem{question}{Question}

\newtheorem{remark}[theorem]{Remark}
\newcommand{\trace}{\operatorname{tr}}
\theoremstyle{plain}
\theoremstyle{remark}

\newtheorem*{keyword}{Key words and Phrases}

\begin{document}

\title[compact quasinilpotent operators]{
On Simultaneous Triangularization of Matrices and quasinilpotency of commutator of compact operators}

\author{Sasmita Patnaik*}
\address{Department of Mathematics and Statistics, Indian Institute of Technology, Kanpur 208016 (INDIA)}
\email{sasmita@iitk.ac.in}

\author{Rahul Sethi**}
\address{School of Mathematics,
Georgia Institute of Technology,
686 Cherry Street
Atlanta, GA 30332-0160 (USA)}
\email{rsethi43@gatech.edu}

\thanks{*Supported by Science and Engineering Research Board, Core Research Grant 002514.\quad\\ ** Partially supported by Science and Engineering Research Board, Core Research Grant 002514.}


\maketitle

\begin{abstract}
In this paper we determine a sufficient condition for the quasinilpotency of a commutator of compact operators via block-tridiagonal matrix form associated with a compact operator. We also prove that every compact operator is unitarily equivalent to the sum of a compact quasinilpotent operator and a triangularizable compact operator. \end{abstract}

 \begin{keyword}
Block-tridiagonal matrix forms, Compact operator, Commutator, Spectrum, Quasinilpotent operator.
\end{keyword}

\section{Introduction}

The block-tridiagonal matrix forms associated with bounded operators continue to be an invaluable tool in operator theory.
A well-known result is that every bounded operator on an infinite-dimensional complex Hilbert space admits a block-tridiagonal matrix form whose diagonal blocks are growing exponentially in size \cite[Theorem 20.4]{PPW}. In this paper, we provide two applications of the block-tridiagonal matrix form associated with a compact operator, one in the study of commutator of compact operators, and the other in the study of the structure of a compact operator. We use simple techniques in our results and our work further emphasizes the role of block-tridiagonal matrix forms associated especially with compact operators. 

In Section 2, we offer a matricial view of investigating questions on the quasinilpotency of commutators of compact operators. We do this by connecting the finite matrices that appear in the block-tridiagonal matrix form of a compact operator with the simultaneous triangularization of pairs of finite matrices. We believe this matrix approach has the potential to offer a better understanding of the structure of commutators, in general.
In Section 3, we investigate the structure of a compact operator by slicing its block-tridiagonal matrix form appropriately and discuss the consequences of it. At the end, we make some remarks which reflect that the block-tridiagonal matrix form associated with a compact operator allows one to dissect the compact operator into some nice pieces which exhibit certain universal phenomena that occur in certain parts of it, if not fully.

Our goal in this paper is twofold: link the quasinilpotency of commutators of compact operators with the simultaneous triangularization of finite matrices and study the impact of the block-tridiagonal matrix form on the structure of a compact operator.

We shall use matrix and bounded operator interchangeably depending on the context.

\section{On quasinilpotency of a commutator of compact operators}

In this section, we provide an application of the simultaneous triangularization of finite matrices in the study of commutator of compact operators via the block-tridiagonal matrix form associated with a compact operator. We determine sufficient conditions on $C$ and $Z$ for which the commutator $CZ-ZC$ is quasinilpotent (see Theorem \ref{ST}). The block-tridiagonal matrix form associated with a compact operator with each nonzero block entry being a finite matrix allowed us to implement the theory of simultaneous triangularization of finite matrices in the study of commutators of compact operators. We provide a computationally more tractable approach to investigate when a commutator of compact operators is quasinilpotent.

We first recall the well-known result of A. Albert and B. Muckenhoupt on commutator of finite matrices:
\cite[Theorem]{AM} 
$$\text{For } C\in M_n(\mathbb C), \text{ if } \trace(C)=0, \text{then there exist matrices A and B such that } AB -BA = C.$$

An immediate consequence is: Every nilpotent matrix is a commutator of matrices. \\

For operators on an infinite-dimensional Hilbert space, compact operators are the direct generalization of finite matrices, so a natural question for the compact counterpart is as follows: \textit{is every compact nilpotent operator a commutator of compact operators?} An affirmative answer to the above question stems from a more than half-a-century old open question of Pearcy and Topping posed in 1971: Is every compact operator a commutator of compact operators? In 2017, Ken Dykema and Amudhan Krishnaswamy Usha proved that every compact nilpotent operator is a commutator of compact operators \cite[Theorem 3.2]{DK}. Their work involved the following structure theorem for nilpotents which they modified for entries to act of the same space in \cite[Lemma 2.2]{DK}: for every nilpotent $T \in B(\h)$, i.e., $T^n = 0$ for some $n>1$, there exists a decomposition $\h = \oplus^n_{i=1}\mathcal H_i$ of the Hilbert space $\h$ into the direct sum of an orthogonal family of subspaces $\mathcal H_1, \mathcal H_2, \dots, \mathcal H_n $ such that the matrix of $T$ corresponding to this decomposition has the upper triangular form \cite[Theorem 1]{H71}. This finite block matrix model allowed them to deal with only finitely many operator equations to construct appropriate compact operators $C$ and $Z$ for which $CZ-ZC = T$, though the computations are quite intricate. 


In the case of operators acting on finite-dimensional Hilbert space, where spectrum of an operator being $\{0\}$ implies that the operator must be nilpotent.
In contrast, there are operators acting on infinite-dimensional Hilbert spaces whose spectrum is $\{0\}$, but they are not nilpotent. Such operators are called  quasinilpotent operators. 

For the class of compact quasinilpotent operators that are not nilpotent, we ask the same question as posed earlier for compact nilpotent operators. Our study is centered around the following question.

\begin{question*}\label{QN-0}
Is every compact quasinilpotent operator a commutator of compact operators?
\end{question*}

Unlike the case of nilpotent c

operator, one does not have a finite matrix structure for compact quasinilpotent operators, so addressing such a question in full generality seems intractable! In order to make any headway, alternatively, we pose the following question:


\begin{question}\label{QN-2}
Determine sufficient condition(s) on compact operators $C$ and $Z$ for which the commutator $CZ-ZC$ is a compact quasinilpotent operator. In other words, when is the commutator of compact operators a quasinilpotent operator?
\end{question}


In the literature, some sufficient conditions on compact operators $C$ and $Z$ are provided to answer Question \ref{QN-2}. For instance, if the pair of compact operators $\{C, Z\}$ is simultaneously triangularizable, then $[C, Z]$ is quasinilpotent \cite[Theorem 7.3.3]{HP}; if every word in $CZ$ and $ZC$ is quasinilpotent, then $[C, Z]$ is quasinilpotent (a consequence of Turovskii's theorem \cite[Corollary 8.1.14]{HP} applied to $CZ$ and $ZC$). However, showing the existence of a maximal chain of subspaces invariant under both $C$ and $Z$ for the pair $\{C, Z\}$ to be simultaneously triangularizable or determining if every word in $CZ$ and $ZC$ is quasinilpotent seems to us a difficult task at hand.

Our motivation to answer Question \ref{QN-2} comes from the answer to its finite-dimensional counterpart, namely, if the finite matrices $C$ and $Z$ are simultaneously triangularizable, then $CZ-ZC$ is a strictly upper triangular matrix and hence is nilpotent. We combined this observation and the continuity of spectrum on the set of compact operators to answer Question \ref{QN-2}.     
We determine a sufficient condition involving finite matrices that yields quasinilpotency of the commutator $[C,Z]$ in Theorem \ref{ST}.

A fundamental tool we use throughout in this paper is the block-tridiagonal matrix form associated with a bounded operator given in \cite[Theorem 20.4]{PPW} (see also \cite[Remark 2.2]{L}), which can be extended to a finite set of bounded operators. We state this result here for the set of two bounded operators and in Section 3 for a single bounded operator (see Equation (\ref{single})).

\begin{theorem*} For $A,B \in \B(\h)$, there is an orthonormal basis $\{f_n\}$ with respect to which both $A$ and $B$ have matrices in the block-tridiagonal form. That is, there is a unitary operator $U$ such that for each of them $U^*AU$ and $U^*BU$ simultaneously the central blocks have block size sequence $\left<k_n\right>$ with $k_1 = 1$ and $k_n = 4(5^{n-2})$ for $n>1$. That is, with respect to $\{f_n\}$, the block-tridiagonal matrix forms of $C := U^*AU$ and $Z :=U^*BU$ are given by
\begin{equation}\label{matrix}
C = \begin{pmatrix} C_1 & A_1 & \mathbf{0} & \ldots \\ B_1 & C_2 & A_2 & \ddots\\ \mathbf{0} & B_2 & C_3 & \ddots\\ \vdots & \ddots & \ddots &\ddots\end{pmatrix}
 \quad\text{and} \quad Z = \left(\begin{matrix} Z_1 & X_1 & \mathbf{0} & \ldots \\ Y_1 & Z_2 & X_2 & \ddots\\ \mathbf{0} & Y_2 & Z_3 & \ddots\\ \vdots & \ddots & \ddots &\ddots\end{matrix}\right).
\end{equation}
The central blocks $C_n$ and $Z_n$ are of size $k_n \times k_n$, where $k_1 = 1$ and $k_n = 4(5^{n-2})$ for $n \ge 2$. The blocks $\{A_n\}$ and $\{X_n\}$ are of size $k_n \times 5k_n$, and the blocks $\{B_n\}$ and $\{Y_n\}$ are of size $5k_n \times k_n$. 
\end{theorem*}

We next introduce a few notations that are required to state Theorem \ref{ST}. Let $C$ and $Z$ be compact operators on $\h$. We denote by $\{f_n\}$ the orthonormal basis with respect to which the matrix forms of $C$ and $Z$ are given by Equation (\ref{matrix}). Let $P_n$ be the projection onto the subspace spanned by $\{f_1, f_2, \ldots, f_{1 + 4 + \cdots + 4(5^{n-2})}\}$ for $n>1$ and $P_1$ the projection onto the span of $\{f_1\}$. Then, for each $n\geq 1$, the matrix forms for $C'_n := P_nCP_n$ and $Z'_n:= P_nZP_n$ are respectively given by 
{\tiny{
\begin{equation}\label{tmatrix}
C'_n = \left(\begin{matrix} C_1 & A_1 & \mathbf{0} & \mathbf{0} &\ldots &\ldots &\ldots \\ 
B_1 & C_2 & A_2 & \mathbf{0} & \ldots &\ldots & \ldots\\ 
\mathbf{0}& \ddots & \ddots & \ddots & \ddots &\ddots & \ddots\\ 
\mathbf{0} & \ddots & B_{n-2} & C_{n-1} & A_{n-1} & \mathbf{0}&\ldots\\
\vdots & \mathbf{0} &\ddots & B_{n-1} & C_n &\mathbf{0} &\cdots\\
\vdots & \mathbf{0} & \mathbf{0} &\mathbf{0} & \mathbf{0} &\mathbf{0} &\ldots\\
\vdots &\vdots &\vdots&\vdots &\ddots &\ddots &\ddots
\end{matrix}\right) \quad \text{and} \quad Z'_n = \left(\begin{matrix} Z_1 & X_1 & \mathbf{0} & \mathbf{0} &\ldots &\ldots &\ldots \\ 
Y_1 & Z_2 & X_2 & \mathbf{0} & \ldots &\ldots & \ldots\\ 
\mathbf{0}& \ddots & \ddots & \ddots & \ddots &\ddots & \ddots\\ 
\mathbf{0} & \ddots & Y_{n-2} & Z_{n-1} & X_{n-1} & \mathbf{0}&\ldots\\
\vdots & \mathbf{0} &\ddots & Y_{n-1} & Z_n &\mathbf{0} &\cdots\\
\vdots & \mathbf{0} & \mathbf{0} &\mathbf{0} & \mathbf{0} &\mathbf{0} &\ldots\\
\vdots &\vdots &\vdots&\vdots &\ddots &\ddots &\ddots
\end{matrix}\right).
\end{equation}
}}
Since $C$ and $Z$ are compact operators, it follows from \cite[Theorem 2]{BG} that $$\|C_n\|, \|B_n\|, \|A_n\|, \|X_n\|, \|Y_n\|, \|Z_n\| \rightarrow 0 \text{ as } n \rightarrow \infty,$$ which further implies that $\|C-C'_n\| \rightarrow 0$ and $\|Z-Z'_n\| \rightarrow 0$ as $n\rightarrow \infty$. We call $C'_n$ and $Z'_n$ the \textit{block-tridiagonal finite matrix} pieces of $C$ and $Z$ respectively, for $n\geq 1$. It immediately follows that $\|[C, Z] - [C'_n, Z'_n]\| \rightarrow 0$ as $n\rightarrow \infty$. We call this a method of approximating $[C, Z]$ via \textit{its block-tridiagonal finite matrix pieces}. We emphasize for the reader that the finite-rank approximations of a compact operator involved here are sparse matrices (matrices with lots of zeros) and have rigid block-tridiagonal matrix forms, which may not be guaranteed in the matrix representation of a bounded operator with respect to any given orthonormal basis.

We re-write $C'_n$ and $Z'_n$ as $2 \times 2$ operator matrices, where the upper left corner of each operator matrix is a square matrix of size $5^{n-1}$ for $n>1$ and the other entries of the matrix are zero operators. That is, with respect to the Hilbert space decomposition $\mathcal H = \mathcal H_{5^{n-1}} \oplus (\mathcal H_{5^{n-1}})^{\perp}$,

\begin{equation*}\label{E1}
C'_n = \left(\begin{matrix} 
C''_n & \mathbf{0}\\
\mathbf{0} & \mathbf{0}
\end{matrix}\right) \quad \text{and} \quad Z'_n = \left(\begin{matrix} 
Z''_n & \mathbf{0}\\
\mathbf{0} & \mathbf{0}
\end{matrix}\right),\end{equation*}
where 
\begin{equation*}\label{tmatirx}
C''_n =\left(\begin{matrix} 
C_1 & A_1 & \mathbf{0} &\ldots &\mathbf{0} \\ 
B_1 & C_2 & A_2 & \mathbf{0} & \vdots \\ 
\mathbf{0}& \ddots & \ddots & \ddots & \mathbf{0} \\ 
\vdots & \ddots & B_{n-2} & C_{n-1} & A_{n-1} \\
\mathbf{0} & \mathbf{0} &\ddots & B_{n-1} & C_n 
\end{matrix}\right) \quad \text{and} \quad Z''_n = \left(\begin{matrix} 
Z_1 & X_1 & \mathbf{0} &\ldots &\mathbf{0} \\ 
Y_1 & Z_2 & X_2 & \mathbf{0} & \vdots \\ 
\mathbf{0}& \ddots & \ddots & \ddots & \mathbf{0} \\ 
\vdots & \ddots & Y_{n-2} & Z_{n-1} & X_{n-1} \\
\mathbf{0} & \mathbf{0} &\ddots & Y_{n-1} & Z_n 
\end{matrix}\right). 
\end{equation*}

It is worth noting that $[C, Z]$ is quasinilpotent if and only if $[U^*CU, U^*ZU]$ is quasinilpotent for each unitary $U \in \mathcal U(\mathcal H)$. So, it suffices to address Question \ref{QN-2} when both $C$ and $Z$ have block-tridiagonal matrix form with respect to a fixed orthonormal basis $\{f_n\}$. With the notations as above, we are now ready to state our theorem. 

\begin{theorem}\label{ST}
 Let $C$ and $Z$ be compact operators, where both simultaneously have block-tridiagonal matrix form given in Equation (\ref{matrix}). If each pair $\{C''_n,Z''_n\}$ is simultaneously triangularizable for $n\ge 1$, then $CZ-ZC$ is quasinilpotent.
\end{theorem}
\begin{proof}
 Since each pair $\{C''_n,Z''_n\}$ is simultaneously triangularizable for $n\ge 1$, there exist unitary matrix $U_n \in M_{5^{n-1}}(\mathbb C)$ such that $U^{*}_nC''_nU_n$ and $U^{*}_nZ''_nU_n$ are both in upper triangular form. This further implies that 
  \begin{equation*}\label{E-1}
(U^{*}_n \oplus I)C'_n(U_n \oplus I) = \left(\begin{matrix} 
U^{*}_nC''_nU_n & \mathbf{0}\\
\mathbf{0} & \mathbf{0}
\end{matrix}\right) \quad \text{and} \quad (U^{*}_n \oplus I)Z'_n(U_n \oplus I) = \left(\begin{matrix} 
U^{*}_nZ''_nU_n & \mathbf{0}\\
\mathbf{0} & \mathbf{0}
\end{matrix}\right),\end{equation*} 
 where $I$ is the identity operator on the orthogonal complement of the finite-dimensional subspace $\mathcal H_n$ of $\h$ of dimension $5^{n-1}$. Therefore, 

 \begin{equation*}\label{E-2}
C'_n = (U_n \oplus I)\left(\begin{matrix} 
U^{*}_nC''_nU_n & \mathbf{0}\\
\mathbf{0} & \mathbf{0}
\end{matrix}\right)(U^{*}_n \oplus I) \quad \text{and} \quad Z'_n = (U_n \oplus I)\left(\begin{matrix} 
U^{*}_nZ''_nU_n & \mathbf{0}\\
\mathbf{0} & \mathbf{0}
\end{matrix}\right)(U^{*}_n \oplus I).\end{equation*}  
 This implies that the commutator 
 
 \begin{equation*}\label{E-3}
[C_n',Z_n'] = C'_nZ'_n - Z'_nC'_n = (U_n \oplus I)\left(\begin{matrix} 
U^{*}_n(C''_nZ''_n - Z''_nC''_n)U_n & \mathbf{0}\\
\mathbf{0} & \mathbf{0}
\end{matrix}\right)(U^{*}_n \oplus I) \end{equation*}  
is nilpotent because $U^{*}_n(C''_nZ''_n - Z''_nC''_n)U_n$  is a strictly upper triangular finite matrix and so the matrix on the left-hand side of the above latter equality is nilpotent.  Moreover, since $\|C_n' - C\| \rightarrow 0$ and  $\|Z_n' - Z\| \rightarrow 0$ as $n\rightarrow \infty$, we have $\|[C_n',Z_n'] - [C,Z]\| \rightarrow 0$ as $n \rightarrow \infty$. By the continuity of the spectrum on compact operators, $\sigma([C_n',Z_n']) \to \sigma([C,Z])$. However, $\sigma([C_n',Z_n']) = \{0\}$ for each $n\ge 1$, and so $\sigma([C,Z]) = \{0\}$, i.e., $[C,Z]$ is quasinilpotent.
\end{proof}

\begin{remark}
(i) Observe that the quasinilpotency of the commutator $[C, Z]$ is determined by the simultaneous triangularization of the finite matrix pieces of $C$ and $Z$ given in Equation (\ref{tmatrix}). The advantage of this approach is that there are several efficient and tractable ways of determining the simultaneous triangularizability of a pair of finite matrices, for instance, see \cite[Chapters 1-4]{HP} that includes many classical theorems such as Engel, McCoy, Levitzki, Kolchin, and Kaplansky theorems, to name a few.\\
(ii) In the special case when $B_n = Y_n = 0$ for $n\geq1$, it suffices to have each pair $\{C_n, Z_n\}$ simultaneously triangularizable for $[C, Z]$ to be quasinilpotent. This is an easy example of compact operators $C$ and $Z$ for which the hypothesis of Theorem \ref{ST} holds.
\end{remark}

The converse of Theorem \ref{ST} does not hold in general. Towards this, we construct $C,Z \in \K(\h)$, which are in block-tridiagonal matrix forms given in Equation (\ref{matrix}) such that $CZ-ZC$ is quasinilpotent, but the pair $\{C''_n,Z''_n\}$ is not triangularizable for some $n\in \N$. To do so, we first observe the following: for each $n\ge 3$, one can construct matrices $A_n,B_n \in M_n(\C)$ and a polynomial $p_n(x,y) \in \C[x,y]$ such that $[A_n,B_n]$ is nilpotent, but $p_n(A_n,B_n)[A_n,B_n]$ is not nilpotent. Indeed, take
$$A_n = \begin{pmatrix}
    0 & 1 & 0 & \cdots & 0 \\
    0 & 0 & 1 & \cdots & 0 \\
    \vdots & \vdots & \vdots & \ddots & \vdots \\
    0 & 0 & 0 & \cdots & 1 \\
    0 & 0 & 0 & \cdots & 0 \\
\end{pmatrix}_{n\times n} \quad\quad B_n = \begin{pmatrix}
    0 & 0 & 0 & \cdots & 0 \\
    0 & 0 & 0 & \cdots & 0 \\
    \vdots & \vdots & \vdots & \ddots & \vdots \\
    0 & 0 & 0 & \cdots & 0 \\
    1 & 0 & 0 & \cdots & 0 \\
\end{pmatrix}_{n\times n}.$$
Then,
\begin{equation*}
A_nB_n-B_nA_n = \begin{pmatrix}
    0 & 0 & 0 & \cdots & 0 \\
    \vdots & \vdots & \vdots & \ddots & \vdots \\
    0 & 0 & 0 & \cdots & 0 \\
    0 & 0 & 0 & \cdots & 0 \\
    1 & 0 & 0 & \cdots & 0 \\
    0 & -1 & 0 & \cdots & 0 \\
\end{pmatrix}_{n\times n} \text{is nilpotent, but }
A_n^{n-2}(A_nB_n-B_nA_n) = \begin{pmatrix}
    1 & 0 & 0 & \cdots & 0 \\
    0 & -1 & 0 & \cdots & 0 \\
    0 & 0 & 0 & \cdots & 0 \\
    \vdots & \vdots & \vdots & \ddots & \vdots \\
    0 & 0 & 0 & \cdots & 0 \\
\end{pmatrix}_{n\times n}
\end{equation*}
is clearly not nilpotent.  It is easy to see that $\|A_n\| = 1= \|B_n\|$ for all $n\ge 3$. Define 
\begin{equation}\label{pq}
P_n = \frac{1}{n}A_n  \quad \text{and }\quad  Q_n = \frac{1}{n}B_n
\end{equation} for all $n\ge 3$. So, $\|P_n\| = \frac{1}{n}= \|Q_n\|$, and hence $\|P_n\|, \|Q_n\| \rightarrow 0$ as $n\rightarrow \infty$.
Observe that 
$$P_nQ_n-Q_nP_n = \begin{pmatrix}
    0 & 0 & 0 & \cdots & 0 \\
    \vdots & \vdots & \vdots & \ddots & \vdots \\
    0 & 0 & 0 & \cdots & 0 \\
    0 & 0 & 0 & \cdots & 0 \\
    \frac{1}{n^2} & 0 & 0 & \cdots & 0 \\
    0 & \frac{-1}{n^2} & 0 & \cdots & 0 \\
\end{pmatrix}_{n\times n},$$
which is nilpotent, but for $p_n(x, y) = x^{n-2}$,
\begin{equation}\label{notnil}
p_n(P_n,Q_n)[P_n,Q_n] = P_n^{n-2}(P_nQ_n-Q_nP_n) = \begin{pmatrix}
    n^{-n} & 0 & 0 & \cdots & 0 \\
    0 & -n^{-n} & 0 & \cdots & 0 \\
    0 & 0 & 0 & \cdots & 0 \\
    \vdots & \vdots & \vdots & \ddots & \vdots \\
    0 & 0 & 0 & \cdots & 0 \\
\end{pmatrix}_{n\times n} \text{is not nilpotent}.
\end{equation}

We are now ready to construct $C$ and $Z$ for which $CZ-ZC$ is quasinilpotent, but the pair $\{C''_n,Z''_n\}$ is not triangularizable for some $n\in \N$.  

\begin{example} 
For $k_1 = 1$ and $k_n = 4(5^{n-2})$ for each $n\ge 2$, let 
$$C = 
\begin{pmatrix}
C_1 & 0 & 0 & \cdots \\
0 & C_2 & 0 & \cdots \\
0 & 0 & C_3 & \cdots \\
\vdots & \vdots & \vdots & \ddots \\
\end{pmatrix} \quad \text{and} \quad Z = 
\begin{pmatrix}
Z_1 & 0 & 0 & \cdots \\
0 & Z_2 & 0 & \cdots \\
0 & 0 & Z_3 & \cdots \\
\vdots & \vdots & \vdots & \ddots \\
\end{pmatrix},$$
where $C_n := P_{k_n}$ and $Z_n := Q_{k_n}$ for $n\ge 1$. 

We first prove that $[C, Z]$ is a compact quasinilpotent operator. Clearly, as $\|P_{k_n}\|, \|Q_{k_n}\| \rightarrow 0$ as $n\rightarrow \infty$, we have $\|C_n\|, \|Z_n\| \rightarrow 0$ as $n\rightarrow \infty$.  Therefore, it follows from \cite[Theorem 2]{BG} that $C$ and $Z$ are compact operators. 

By construction, $[C_n,Z_n]$ is nilpotent for each $n\ge 1$ and $[C_n, Z_n] \rightarrow 0$ as $n\rightarrow \infty$. The commutator $CZ-ZC$ is given by 
$$[C,Z] = \begin{pmatrix}
[C_1, Z_1] & 0 & 0 & \cdots \\
0 & [C_2, Z_2] & 0 & \cdots \\
0 & 0 & [C_3, Z_3] & \cdots \\
\vdots & \vdots & \vdots & \ddots \\
\end{pmatrix}.$$
For each $n\geq 1$, the truncated sequence of nilpotent operators constructed from the above matrix form of $[C, Z]$ is defined as  $$[C,Z]_n :=  \begin{pmatrix}
[C_1, Z_1] & 0 & 0 & \cdots & 0 &0 &\ldots \\
0 & [C_2, Z_2] & 0 & \cdots & 0 & 0 &\ldots \\
\vdots & \vdots & \vdots & \ddots & \vdots & \vdots & \ldots \\
0 & 0 & 0 & \cdots & [C_n, Z_n] & 0 & \ldots\\
0 & 0 & 0 & \ldots & 0 & 0 & \ddots \\
\vdots & \vdots & \vdots & \vdots & \vdots & \ddots & \ddots
\end{pmatrix}.$$
A straightforward computation shows that $\|[C, Z]_n - [C, Z]\| \rightarrow 0$ as $n \rightarrow \infty$. Moreover, since the norm limits of compact quasinilpotents (here in particular, $[C, Z]_n$ are nilpotents) is again a compact quasinilpotent \cite[Corollary 7.2.11]{HP}, the commutator $[C, Z]$ is a compact quasinilpotent operator. 

Following the matrix notation from Equation (\ref{tmatrix}), we have
$$C_n' = \begin{pmatrix}
C_1 & 0 & 0 & \cdots & 0 &0 &\ldots \\
0 & C_2 & 0 & \cdots & 0 & 0 &\ldots \\
0 & 0 & C_3 & \cdots & 0 & 0 &\ldots\\
\vdots & \vdots & \vdots & \ddots & \vdots & \vdots & \ldots \\
0 & 0 & 0 & \cdots & C_n & 0 & \ldots\\
0 & 0 & 0 & \ldots & 0 & 0 & \ddots \\
\vdots & \vdots & \vdots & \vdots & \vdots & \ddots & \ddots
\end{pmatrix} \quad \text{and} \quad 
Z_n' = \begin{pmatrix}
Z_1 & 0 & 0 & \cdots & 0 &0 &\ldots \\
0 & Z_2 & 0 & \cdots & 0 & 0 &\ldots \\
0 & 0 & Z_3 & \cdots & 0 & 0 &\ldots\\
\vdots & \vdots & \vdots & \ddots & \vdots & \vdots & \ldots \\
0 & 0 & 0 & \cdots & Z_n & 0 & \ldots\\
0 & 0 & 0 & \ldots & 0 & 0 & \ddots \\
\vdots & \vdots & \vdots & \vdots & \vdots & \ddots & \ddots
\end{pmatrix},$$
where $C_n, Z_n \in M_{k_n}(\mathbb C)$ for each $n\ge 1$. We next show that the pair $\{C''_n,Z''_n\}$ is not triangularizable for each $n\geq 2$. (The target was to find some $n\in \N$ such that $\{C''_n,Z''_n\}$ is not triangularizable. Interestingly, this example yields more than we need - we shall prove that $\{C''_n,Z''_n\}$ is not triangularizable for any $n \ge 2$.) The proof is by contradiction.

Indeed, assume that $\{C''_n,Z''_n\}$ is triangularizable for some $n \ge 2$. Then, by McCoy's Theorem for finite matrices (\cite[Theorem 1.3.4]{HP}), $p(C''_n,Z''_n)[C''_n,Z''_n]$ is nilpotent for every $p(x,y) \in \mathbb C[x,y]$. In particular, choose $p(x,y) = p_{k_n}(x,y) = x^{k_n - 2}$. We recall that $k_1 = 1$ and $k_n = 4(5^{n-2})$ for $n>1$. We have
{\tiny{
\begin{equation*}
p(C'_n,Z'_n)[C'_n,Z'_n] = \begin{pmatrix}
p(C_1,Z_1)[C_1,Z_1] & 0 & 0 & \cdots & 0 &0 &\ldots \\
0 & p(C_2,Z_2)[C_2,Z_2] & 0 & \cdots & 0 & 0 &\ldots \\
0 & 0 & \ddots & \cdots & 0 & 0 &\ldots\\
\vdots & \vdots & \vdots & \ddots & \vdots & \vdots & \ldots \\
0 & 0 & 0 & \cdots & p(C_n,Z_n)[C_n,Z_n] & 0 & \ldots\\
0 & 0 & 0 & \ldots & 0 & 0 & \ddots \\
\vdots & \vdots & \vdots & \vdots & \vdots & \ddots & \ddots
\end{pmatrix} = \left(\begin{matrix} 
 p(C''_n,Z''_n)[C''_n,Z''_n]& \mathbf{0}\\
\mathbf{0} & \mathbf{0}
\end{matrix}\right).
\end{equation*}
}}

Note that for $k_1 = 1$, the $1\times 1$ matrix $p(C_1,Z_1)[C_1,Z_1]$ is always zero and for each $n>1$, the matrix $p(C_n,Z_n)[C_n,Z_n]$ is of size greater than $3$. For convenience, let $T_n := p(C'_n,Z'_n)[C'_n,Z'_n]$ and $S_j := p(C_j,Z_j)[C_j,Z_j]$ for each $1\le j\le n$. Recall that by hypothesis $C_j := P_{k_j}, Z_j := Q_{k_j}$, where $P_{k_j}, Q_{k_j}$ are given by Equation (\ref{pq}). Since the matrix of $T_n$ has finitely many nonzero blocks, it follows from \cite[Chapter 11, \S 98]{H} that for $n>1$,
$$\sigma(T_n) = \bigcup_{j=1}^n \sigma(S_j).$$
Since $p(C''_n,Z''_n)[C''_n,Z''_n]$ is nilpotent, $T_n$ is also nilpotent, and thus we have $\sigma(T_n) = \{0\}$. Hence $\bigcup_{j=1}^n \sigma(S_j) = \{0\}$. This implies that $\sigma(S_j) = \{0\}$ for all $1\le j\le n$. However, for $j>1$, $\sigma(S_j) \ne \{0\}$ as $S_j$ is not nilpotent by construction in Equation (\ref{notnil}) (in there, replace $j$ with $k_j = 4(5^{j-2}) > 3$), and hence a contradiction. Therefore, $\{C''_n,Z''_n\}$ is not triangularizable for all $n\geq 2$. 
\end{example}

\begin{remark} The construction above works more generally when $\{k_n\}_{n=1}^\infty$ is any strictly increasing sequence of natural numbers. In fact, if $k_1 = 1$, then $\{C''_n,Z''_n\}$ is not triangularizable for all $n\geq 2$. If $k_1 > 1$, then $\{C''_n,Z''_n\}$ is not triangularizable for all $n\geq 1$.
\end{remark}

\section{On the structure of a compact operator via its block-tridiagonal matrix form}
In \cite[Theorem 2]{FS}, Fong and Sourour showed that every bounded operator on a Hilbert space is the sum of two quasinilpotent operators if and only if it is not a nonzero scalar plus a compact operator. As a consequence, they showed in \cite[Corollary 2]{FS} that every compact operator is a sum of two compact quasinilpotent operators. 

In this section, we establish an alternate structure theorem for a compact operator through the direct use of the block-tridiagonal matrix form associated with it. 
In the finite-dimensional case, it is well-known that every matrix is unitarily equivalent to an upper triangular matrix, which is the sum of a diagonal matrix and a nilpotent matrix. 
Here we prove an infinite-dimensional analog of this finite-dimensional case. That is, every compact operator is unitarily equivalent to the sum of an (upper) triangularizable compact operator (i.e., whose matrix is upper triangular via a unitary matrix) and a compact quasinilpotent operator (whose matrix is strictly lower triangular). This is discussed below and summarized in Theorem \ref{compact}.

For any compact operator $T$ (more generally, for any bounded operator), by \cite[Theorem 20.4]{PPW}, there is an orthonormal basis $\{f_n\}$ such that with respect to that basis the matrix representation of $K:= W^*TW$ has a block-tridiagonal matrix form, where the central blocks $C_n$ have block size sequence $\{k_n\}$ given by $k_1 = 1$ and $k_n = 2(3^{n-2})$ for $n > 1$. (Here $W$ is the unitary associated with the change of basis.) That is, with respect to the orthonormal basis $\{f_n\}$ obtained in \cite[Theorem 20.4]{PPW}, one has
\begin{equation}\label{single}
K = \left(\begin{matrix} C_1 & A_1 & 0 & \ldots \\ B_1 & C_2 & A_2 & \ddots\\ 0 & B_2 & C_3 & \ddots\\ \vdots & \ddots & \ddots &\ddots\end{matrix}\right). 
\end{equation}

We then split this matrix in two pieces as follows.
\begin{equation}\label{pieces}
K = \left(\begin{matrix} C_1 & A_1 & 0 & \ldots \\ B_1 & C_2 & A_2 & \ddots\\ 0 & B_2 & C_3 & \ddots\\ \vdots & \ddots & \ddots &\ddots\end{matrix}\right) = \underbrace{\left(\begin{matrix} C_1 & A_1 & 0 & \ldots \\ 0 & C_2 & A_2 & \ddots\\ 0 & 0 & C_3 & \ddots\\ \vdots & \ddots & \ddots &\ddots\end{matrix}\right)}_S + \underbrace{\left(\begin{matrix} 0 & 0 & 0 & \ldots \\ B_1 & 0 &  0 & \ddots\\ 0 & B_2 & 0 & \ddots\\ \vdots & \ddots & \ddots &\ddots\end{matrix}\right)}_Q,
\end{equation}
where the``first piece" of $K$ denoted by $S$ is an upper triangularizable compact operator and the ``second piece" of $K$ denoted by $Q$ is a compact quasinilpotent operator. Indeed, since $K$ is a compact operator, it follows from \cite[Theorem 2]{BG} that  $\|A_n\|, \|B_n\|, \|C_n\| \rightarrow 0$ as $n \rightarrow \infty$. Since  $\|B_n\| \rightarrow 0$ as $n \rightarrow \infty$, so $Q$ is a compact operator. For each $n\geq 1$, let
\begin{equation*}\label{nil}
F_n := \left(\begin{matrix} 0 & 0 & 0 & \ldots & \ldots& \ldots & \ldots \\ 
B_1 & 0 &  0 & \ddots & \ddots & \ddots & \ddots\\ 
0 & B_2 & 0 & \ddots & \ddots & \ddots & \ddots\\
 \vdots & \ddots & \ddots &\ddots & \ddots & \ddots & \ddots \\
 \vdots & \ddots & \ddots & B_n &\ddots & \ddots & \ddots \\
 \vdots & \ddots & \ddots & \ddots &0 &\ddots & \ddots\\
 \vdots & \ddots & \ddots &\ddots &\ddots&\ddots&\ddots
 \end{matrix}\right).
\end{equation*}
Note that $\|Q- F_n\| \rightarrow 0$ as $n \rightarrow \infty$, where $F_n$'s are finite-rank nilpotent operators.  And as the norm limit of nilpotent operators is quasinilpotent, therefore $Q$ is quasinilpotent. 

Consider the compact operator $S$ in Equation (\ref{pieces}). Since $C_n$ is a square matrix for each $n\geq 1$, $C_n$ is upper triangularizable. That is, for each $n\geq 1$, there exists a unitary matrix $U_n$ such that $U_n^*C_nU_n = \Delta_n$, where $\Delta_n$ is an upper triangular matrix. Consider the unitary operator $U := \oplus U_n$ on the Hilbert space $\h$. Then, a straightforward computation shows that $U^*SU$ is an upper triangular matrix. Therefore, $S$ represents a compact operator whose matrix is upper triangularizable. Let $\Delta :=U^*SU$. Then the matrix form associated with $\Delta$ is given by


\begin{equation}\label{triangle}
\Delta = \left(\begin{matrix} \Delta_1 & A'_1 & 0 & \ldots \\ 0 & \Delta_2 & A'_2 & \ddots\\ 0 & 0 & \Delta_3 & \ddots\\ \vdots & \ddots & \ddots &\ddots\end{matrix}\right).
\end{equation}

Consider the same unitary operator $U$ constructed above using the $U_n$'s.
Since $K = S + Q$ (Equation (\ref{pieces})), 

\begin{equation}\label{compact_pieces}
U^*KU =  \left(\begin{matrix} \Delta_1 & A'_1 & 0 & \ldots \\ 0 & \Delta_2 & A'_2 & \ddots\\ 0 & 0 & \Delta_3 & \ddots\\ \vdots & \ddots & \ddots &\ddots\end{matrix}\right)  + \left(\begin{matrix} 0 & 0 & 0 & \ldots \\ 
U^*_2B_1U_1 & 0 &  0 & \ddots\\ 
0 & U_3^*B_2U_2 & 0 & \ddots\\ 
\vdots & \ddots & \ddots &\ddots\end{matrix}\right) =  \Delta + U^*QU.
\end{equation}

Since $K = W^*TW$, one has $U^*W^*TWU = \Delta + U^*QU$. Let $U_0 := WU$, which is a unitary operator. Therefore, for a given compact operator $T$, we have
\begin{equation}\label{decom}
U_0^*TU_0 = \Delta + U^*QU,
\end{equation}
where $\Delta$ is an upper triangular matrix representing a compact operator and $U^*QU$ is a quasinilpotent matrix representing a compact quasinilpotent operator.

Recall that in the finite-dimensional case, a finite matrix $A$ is unitarily equivalent to an upper triangular matrix, i.e., $V^*AV = D + Q$, where $D$ is a diagonal matrix (upper triangular matrix) and $Q$ is a nilpotent matrix (strictly lower triangular matrix) for some unitary matrix $V$. In the language of the spectrum of a finite matrix, $V^*AV = D + Q$ implies that $\sigma(V^*AV - Q) = \{0\}$. 

One can view Equation (\ref{decom}) as an infinite-dimensional analog of the above mentioned finite-dimensional matrix decomposition in the following theorem.
 
To summarize:

\begin{theorem}(A structure theorem for compact operators)\\\label{compact}
 Every compact operator $T$ is unitarily equivalent to an upper triangularizable compact operator and a compact quasinilpotent operator. Moreover, following the notation from Equation (\ref{pieces}), $\sigma(U_0^*TU_0 - \Delta) = \{0\}$. In particular, if $B_n = 0$ for $n\geq 1$, then $K$, i.e., $W^*TW$ is triangularizable.
 
 \smallskip
Alternatively, every compact operator is unitarily equivalent to a compact operator $K$, which has a normal part $N$ such that $K-N$ is a sum of two quasinilpotent compact operators. (One can see this by splitting $U^*SU$ in Equation (\ref{triangle}) into the sum of a diagonal matrix and a strictly upper triangular matrix.)
\end{theorem}

\section*{Concluding remarks}
We filter out some observations from Sections 2-3 that are perhaps worth mentioning when we consider the block-tridiagonal matrix form associated with a  compact operator or a pair of compact operators.
In particular, observations (i)-(iii) below reveal that a compact operator is always ``well-behaved" in appropriate sense after removing a certain piece of it.

\begin{enumerate}[label=(\roman*)]
\item Unlike the finite-dimensional case, where every matrix is triangularizable, a compact operator acting on an infinite-dimensional Hilbert space need not be triangularizable. But, if the compact operator is in its block-tridiagonal matrix form given by Equation (\ref{pieces}), then there is a piece of it, namely, $S$ in Equation (\ref{pieces}), which is always (upper) triangularizable. 

\item In general, a pair of compact operators $\{K_1, K_2\}$ need not be simultaneously triangularizable. But, if $K_1$ and $K_2$ both are simultaneously in block-tridiagonal matrix forms, then by extracting its triangularizing pieces, the pair $\{K_1 - S_1,  K_2 - S_2\}$ is \textit{always} simultaneously (upper) triangularizable. Indeed, the set $\{K_1 - S_1,  K_2 - S_2\} = \{Q_1, Q_2\}$, where $K_1 = S_1 + Q_1, K_2 = S_2 +Q_2$ given by Equation (\ref{pieces}). We observe that $p(Q_1, Q_2)[Q_1,Q_2]$ is again quasinilpotent for each polynomial $p(Q_1, Q_2)$. So, by McCoy's  Theorem \cite[Theorem 7.3.3]{HP}, the pair $\{A, B\}$ is simultaneously triangularizable. \textit{That is, even if $\{K_1, K_2\}$ may not be triangularizable, $\{K_1 - S_1,  K_2 - S_2\}$ is always simultaneously triangularizable.} (This observation can be extended to a finite set of compact operators $\{K_1, \dots, K_N\}$, which follows from \cite[Corollary 7.3.4]{HP}.)

\item The matrix splitting given in Equation (\ref{pieces}) of the block-tridiagonal matrix form associated with a compact operator (or a pair of compact operators) touches upon another question on commutators, namely, on commutators that are trace class with trace zero -- for $C, Z \in B(\mathcal H)$, with $C$ or $Z$ a special operator and that $CZ-ZC$ is trace class, then is $\tr(CZ-ZC) = 0$? (See \cite{K} and references therein.) Because the trace is unitarily invariant and that $T$ is trace class if and only if $U^*TU$ is trace class for each unitary operator $U$, without loss of generality, the above question can be restated for when both $C$ and $Z$ are in block-tridiagonal matrix form. We have the following: for $C$ and $Z$ compact operators each with block-tridiagonal matrix form, if the commutator $[C - S_1,  Z - S_2]$ is trace class, then $\tr([C - S_1,  Z - S_2]) = 0$, where $C = S_1 + Q_1$ and $Z = S_2 + Q_2$. Indeed, $\tr([C - S_1,  Z - S_2]) = \tr(Q_1Q_2 - Q_2Q_1)$ and note that the diagonal entries of $Q_1Q_2 - Q_2Q_1$ are all equal to zero, so the trace is zero.
(In \cite[Theorems 4-5]{K}, F. Kittaneh proved that $\tr(CZ-ZC) = 0$ under certain special sufficient conditions on $C$ and $Z$.) 

\item For a compact operator $K$ with matrix decomposition given in Equation (\ref{pieces}), if for each $n\geq 1$, $C_n$ is a nilpotent matrix, then $K$ is \textit{unitarily} similar to a \textit{zero-diagonal }operator. In this case, $K$ is a sum of two quasinilpotent operators.
\end{enumerate}

\end{document}